\documentclass[11pt, a4paper]{article}

\usepackage{amsmath}                                                
\usepackage{amssymb}                                                
\usepackage{amsthm}                                                 
\usepackage{amsfonts}                                               
\usepackage{mathrsfs}                                               
\usepackage{enumerate}                                          
\usepackage{hyperref}                                               
\usepackage{cite}                                                       
\usepackage{fullpage}                                               
\theoremstyle{plain}
\newtheorem{theorem}{Theorem}[section]
\newtheorem{lemma}{Lemma}[section]

\newtheorem{remark}{Remark}[section]
\newtheorem{condition}{Condition}[section]

\theoremstyle{definition}

\theoremstyle{remark}

\renewcommand{\Re}{\mathbf{Re}}

\title{Frequency Theorem for discrete time stochastic system with multiplicative noise}
\author{Peter Situmbeko Nalitolela\footnote{peter.nalitolela@gmail.com}\\Department of Mathematics, University of Dar Es Salaam\\
                Dar Es Salaam, Tanzania
                \and
                Nikolai Dokuchaev\footnote{N.Dokuchaev@curtin.edu.au}\\ Department of Mathematics \& Statistics, Curtin University\\
                Perth, Australia
                }
\date{}
\begin{document}
\maketitle
\begin{abstract}
In this paper we consider the problem of minimizing a quadratic
functional for a discrete-time linear stochastic system with
multiplicative noise, on a standard probability space, in infinite
time horizon. We show that the necessary and sufficient conditions
for the existence of the optimal control can be formulated as matrix
inequalities in frequency domain. Furthermore, we show that if the
optimal control exists, then certain Lyapunov equations must have a
solution. The optimal control is obtained by solving a deterministic
linear-quadratic optimal control problem whose functional depends on
the solution to the Lyapunov equations. Moreover, we show that under
certain conditions, solvability of the Lyapunov equations is
guaranteed. We also show that, if the frequency inequalities are
strict, then the solution is unique up to equivalence.
\end{abstract}
\section{Introduction}\label{intro}
Kalman--Yakubovich Lemma (KY Lemma) was a groundbreaking result that
paved way for a solutions to lots of problems in control theory,
including optimal control. The first variant of the Lemma was
derived by Yakubovich in 1962 (see \cite{Yakubovich62}). The
following year, the discrete-time version of that result was derived
by Szeg\H{o} and Kalman (see \cite{Szego_Kalman63}). It is called
sometimes the Kalman--Szeg\H{o} Lemma (KS Lemma); see
\cite{Likhtarnikov_et_al96, Pakshin06, Barabanov07} for a
comprehensive review of various results in control theory derived
from the KY Lemma. Various works such as \cite{Andreev72,
Andreev_et_al71, Andreev_et_al72, Andreev_Shepelyavyi72} considered
problems with quadratic functionals whereas Yakubovich (see
\cite{Yakubovich74,Yakubovich75}) derived the KY Lemma for the case
in which both the control and state vectors are both Hilbert spaces.
\par Dokuchaev \cite{Dokuchaev84} considered a
continuous time stochastic linear-quadratic optimal control problem,
with the state evolution described by It\^{o} equations, with state
dependent coefficients; a generalization of the Frequency Theorem
was obtained. We consider a discrete-time analogy of the problem
studied in \cite{Dokuchaev84}. We show that the necessary and
sufficient conditions for the existence of the optimal control can
be formulated as matrix inequalities in frequency domain.
Furthermore, we show that if the optimal control exists, then
certain Lyapunov equations must have a solution. The optimal control
is obtained by solving a deterministic linear-quadratic optimal
control problem whose functional depends on the solution to the
Lyapunov equations. Moreover, we show that under certain conditions,
solvability of the Lyapunov equations is guaranteed. We also show
that, if the frequency inequalities are strict, then the solution is
unique up to equivalence.
\section{Problem Statement}\label{problem}
We consider the following optimization problem on a standard
probability space, $(\Omega,\mathscr{F}, \mathbf{P})$.\par \noindent
\begin{eqnarray}
\Phi \left(u_{.}\right) & = & \sum\limits_{t=0}^{+\infty}
\hbox{Minimize}\quad \mathbf{E} \, \left[ x_{t}^{*} G x_{t} + 2\Re
\, x_{t}^{*} \gamma u_{t} + u_{t}^{*} \Gamma
u_{t}\right]\label{problem_functional}
\end{eqnarray}
over the set
\begin{eqnarray}
\mathbf{U} & = & \left\{ u_{t} \in \mathbb{R}^{m}: \sum\limits_{t=0}^{+\infty} \left|u_{t}\right|^{2} < +\infty\right\}\label{problem_feasible}
\end{eqnarray}
subject to
\begin{eqnarray}
x_{t+1} &=& A x_{t} + b u_{t} + C x_{t} \xi_{t+1}, \qquad t=0,1,2,\ldots \label{problem_constraints}\\
x_{0} &=& a. \label{problem_IC}
\end{eqnarray}
Here $x_{t}$ is a random $n$-vector of states, $u_{t}$ is an
$m$-vector of controls and $\mathbf{U}$ is the set of admissible
controls.  Matrices $A\in R^{n\times n}$, $b\in R^{n\times m}$,
$C\in R^{n\times n}$, $G=G^\top\in R^{n\times n}$, $\gamma\in
R^{n\times n} $, and $\Gamma=\Gamma^\top\in{m\times m}$ are
constant. The scalar $\xi_{t}\in \mathbb{R}$ is the discrete-time
white noise adapted to a flow of non-decreasing $\sigma$-algebras
$\mathscr{F}_{t}\subset \mathscr{F}$ such that $\mathbf{E} \xi_{t} =
0$, $\mathbf{Var} \left(\xi_{t}\right) = 1$. The vector $a$ is
random, measurable with respect to $\mathscr{F}_{0}$, independent of
$\{\xi_{t}\}_{t=0}^{+\infty}$ and is such that $\mathbf{E}\left|a
\right|^{2} < +\infty$ and $\mathbf{E}\left|a a^{\top}\right|^{2} <
+\infty$;  we denote by $|.|$ the Euclidean norm for vectors and
Frobenius norm for matrices.

We assume all the matrices in \eqref{problem_feasible} and
\eqref{problem_constraints} are real and we restrict our
considerations to the case when all eigenvalues
$\lambda\left(A\right)$ of $A$ lie inside the unit disk on the
complex plane (that is, the spectral radius of $A$ is
$\rho\left(A\right) < 1$). 
Moreover, we  assume that the system is stable in mean-square sense
for $u_{t}\equiv 0$. Various sufficient conditions of this stability
can be found in \cite{Korenevskii86_2, Korenevskii86_3,
Korenevskii92, Korenevskii98, Korenevskii00, Levit_Yakubovich72,
Morozan68, Ryashko_Schurz96, Willems73, Pakshin06} and other works.
\par For random $x_{t}, y_{t} \in \mathbb{C}^{n}$ we denote the inner product $\left(x_{.},
y_{.}\right)$ by $\left(x_{.}, y_{.}\right) =
\sum\limits_{t=0}^{+\infty} \mathbf{E} \, x_{t}^{\top}
\overline{y}_{t}$ and the norm by $\left\|x_{.}\right\| =
\sqrt{\left(x_{.}, x_{.}\right)}$. Furthermore, we write
$\left\|x_{.}\right\|_{1}=\sum\limits_{t=0}^{+\infty} \mathbf{E} \,
\left|x_{t}\right|_{1}$ where $|x|_{1}$ is the $l_{1}$-norm
$|x|_1=\sum_i|x_i|$ of a vector $x$ or an entrywise $l_{1}$-norm
$|x|_1=\sum_{ij}|x_{ij}|$ of a matrix $x$.
\section{Main Results}\label{results}
\begin{condition}\label{matrix_condition}
There exist symmetric matrices $H$ and $\Theta$ in
$\mathbb{C}^{n\times n}$ satisfying
\begin{eqnarray}
A^{\top}H A - H + \Theta &=& 0 ,\label{matrix_eqn1}\\
\Theta - C^{\top}H C - G &=& 0 .\label{matrix_eqn2}
\end{eqnarray}
\end{condition}
\noindent Let $\Theta$ be the matrix satisfying Condition
\ref{matrix_condition}. Consider the hermitian form $\mathcal{F} :
\mathbb{C}^{n} \times \mathbb{C}^{m} \mapsto \mathbb{R}$ given by
\begin{eqnarray}
\mathcal{F}\left(x,u\right) &=& x^{*}\Theta x + 2\Re x^{*}\gamma u + u^{*}\Gamma u .\label{hermitian_form}
\end{eqnarray}
Let $g : \mathbb{C} \mapsto \mathbb{C}^{n \times n}$ be the matrix-valued function
\begin{eqnarray}
g\left(z\right) &=& \left(z I - A\right)^{-1},\label{transfer_matrix}
\end{eqnarray}
We denote the unit circle by $\mathbf{\zeta} = \left\{z
\in\mathbb{C}: \left|z\right| = 1\right\}$.
\par
The following Theorem establishes necessary and sufficient
conditions for the existence of optimal $u^{o}$ for the problem
\eqref{problem_functional}-\eqref{problem_IC}
\begin{theorem}\label{main_theorem}
If there exists exists a $u^{o}\in \mathbf{U}$ such that $\Phi\left(u^{o}\right)\leq \Phi\left(u\right)$, for all $u\in \mathbf{U}$ then
\begin{enumerate}[i)]
\item it is necessary that
\begin{eqnarray}
\mathcal{F}\left(g\left(z\right) b u,u\right) &\geq& 0, \qquad \left(\forall z \in \mathbf{\zeta}, \forall u \in \mathbb{C}^{m}\right).\label{necessary_condition}
\end{eqnarray}
\item Furthermore, if there exists a $\delta > 0$ such that
\begin{eqnarray}
\mathcal{F}\left(g\left(z\right) b u,u\right) &\geq& \delta \left|u\right|_2^2, \qquad \left(\forall z \in \mathbf{\zeta}, \forall u \in \mathbb{C}^{m}\right),\label{necessary_condition_strict}
\end{eqnarray}
then $u^{o}$ is unique (up to equivalence).
\end{enumerate}
\end{theorem}
Theorem \ref{main_theorem} above is an analog of KS Lemma for
discrete-time optimal stochastic control problem
\eqref{problem_functional}-\eqref{problem_IC}. This is a discrete
time version of a continuous-time result obtained in
\cite{Dokuchaev84} for the case when $\gamma = 0$ and in Chapter 5
of \cite{Nalitolela10}) for the general $\gamma$.
\subsection{Proof of Theorem \ref{main_theorem}}
\begin{lemma}\label{proof_lemma1}
If $u_{t} \in \mathbf{U}$, then $\sup_{t\ge 0}\mathbf{E}\,
\left|x_{t}\right|^{2}<+\infty$ for the solution of system
\eqref{problem_constraints}-\eqref{problem_IC}.
\end{lemma}
\begin{proof}
Let
\begin{eqnarray}
\mu_{t} &=& \mathbf{E} \, x_{t}, \label{mean_x}\\
M_{t} &= & \mathbf{E}\, x_{t}x_{t}^{\top}\label{mean_x_sqd}
\end{eqnarray}
From \eqref{problem_constraints}-\eqref{problem_IC} and \eqref{mean_x}, we have
\begin{eqnarray}
\mu_{t+1} &=& A\mu_{t} +  b u_{t}\qquad t=0,1,2,\ldots,\label{mu_t}\\
\mu_{0} &= & \mathbf{E}\, a.\label{mu_0}
\end{eqnarray}
Note that $\left|\mathbf{E}\, a\right|^{2}
=\sum\limits_{i=1}^{n}\left(\mathbf{E}\, a_{i}\right)^2\leq
\sum\limits_{i=1}^{n}\mathbf{E}\, a_{i}^2=\mathbf{E}\left|\,
a\right|^{2}<+\infty$. Thus, using the fact that
$u_{t}\in\mathbf{U}$ and $\rho\left(A\right) < 1$, it follows from
\eqref{mu_t}-\eqref{mu_0} that $\left\|\mu_{t}\right\| <
+\infty$.\par From \eqref{problem_constraints}-\eqref{problem_IC}
and \eqref{mean_x_sqd}, we have
\begin{eqnarray}
M_{t+1} &=& A M_{t} A^{\top} + A \mu_{t} u_{t}^{\top} b^{\top} + b u_{t} \mu_{t}^{\top} A^{\top}+b u_{t} u_{t}^{\top} b^{\top}+ C M_{t}C^{\top},\label{proof_deterministic1}\\
M_{0} &=& \mathbf{E}\, a a^{\top}.\label{proof_deterministic2}
\end{eqnarray}
Let $Q_{t} = A \mu_{t} u_{t}^{\top} b^{\top} + b u_{t}
\mu_{t}^{\top} A^{\top}+b u_{t} u_{t}^{\top} b^{\top}$. Let us
denote the $j$-th colum of a matrix $D$ by $D^{(j)}$. We define the
vectors $q_{t}, m_{t}\in \mathbb{C}^{n^{2}}$ as
\begin{eqnarray}
q_{t} = \left[
                            \begin{array}{c}
                            Q_{t}^{(1)}\\
                            Q_{t}^{(2)}\\
                            \vdots\\
                            Q_{t}^{(n)}\\
                            \end{array}
                            \right], \qquad
m_{t} = \left[
                            \begin{array}{c}
                            M_{t}^{(1)}\\
                            M_{t}^{(2)}\\
                            \vdots\\
                            M_{t}^{(n)}\\
                            \end{array}
                        \right]. \label{proof_defn_qm}
\end{eqnarray}
The vectors $q_{t}$ and $m_{t}$ are formed by stacking up the
columns of the matrices $Q_{t}$ and $M_{t}$, respectively. Set
$\mathcal{A}=A\otimes A + C\otimes C$ (where $\otimes$ denotes the
Kronecker product). We can then rewrite \eqref{proof_deterministic1}
as
\begin{eqnarray}
m_{t+1} &=& \mathcal{A} m_{t} + q_{t}.\label{matrix_to_linear}
\end{eqnarray}
Note that the system in \eqref{matrix_to_linear} is of dimension
$n^{2}$, however, due symmetry, it can be reduced to a system of
dimension $\dfrac{n^{2}+n}{2}$.
\par The assumption that the system
\eqref{problem_constraints}-\eqref{problem_IC} is stable in the
mean-square sense for $u_{t} = 0$, is equivalent to $m_{t}$ being
stable for $q_{t} = 0$, which is true if and only if the spectral
radius of $\mathcal{A}$ is $\rho\left(\mathcal{A}\right) < 1$. From
the solution of \eqref{matrix_to_linear}, we can show, using
H\"{o}lder's inequality and Young's theorem, that
$\left\|m_{t}\right\|_{1} < +\infty$, therefore $\sup_{t\ge
0}\mathbf{E}\, \left|x_{t}\right|^{2}<+\infty$. This compoletes the
proof of  Lemma \ref{proof_lemma1}.
\end{proof}
It follows from Lemma \ref{proof_lemma1} that the $Z$-transform,
$\hat{x}\left(z\right)$, of $x_{t}$, exists, and it's radius of
convergence contains the unit circle, $\mathbf{\zeta}$. If we set
$x_{t}=0$, $u_{t}=0$ for all $t<0$ we can then take the
$Z$-transform of the system
\eqref{problem_constraints}-\eqref{problem_IC} and obtain
\begin{eqnarray}
\hat{x}\left(z\right) &=& z g\left(z\right) a + g\left(z\right) b
\hat{u}\left(z\right)+g\left(z\right) C
\sum_{t=-\infty}^{\infty}\xi_{t+1}
\frac{x_{t}}{z^{t}}.\label{transfer_function}
\end{eqnarray}
Let $D$ be an $n\times n$ real symmetric matrix and let $T\mapsto
\mathbb{R}^{n \times n} \times \mathbb{R}^{n \times n}$ be defined
by
\begin{eqnarray}
T\left(D\right) &=& \frac{1}{2\pi i} \oint_{\mathbf{\zeta}} C^{\top}g\left(z\right)^{\top} D g(z) C\frac{1}{z}dz. \label{theta_function}
\end{eqnarray}
\begin{lemma}\label{theta_condition}
Condition \ref{matrix_condition} is satisfied if and only if $\Theta$ satisfies
\begin{eqnarray}
G &=& \Theta - T\left(\Theta\right). \label{G_Theta}
\end{eqnarray}
\end{lemma}
\begin{proof}
Suppose there exists a $\Theta$ such that \eqref{G_Theta} holds. Let
\begin{eqnarray}
H &=& \frac{1}{2\pi i} \oint_{\mathbf{\zeta}} g\left(z\right)^{\top} \Theta g(z) \frac{1}{z}dz. \label{matrix_H}
\end{eqnarray}
It follows from Parseval's identity that $H =
\sum\limits_{t=0}^{+\infty}\left(A^{\top}\right)^{t}\Theta
A^{t}=\Theta+A^{\top}H A$. It therefore follows from \eqref{G_Theta}
that $G=\Theta - C^{\top}H C$. Hence Condition
\ref{matrix_condition} is satisfied.\par Conversely, suppose
\eqref{matrix_eqn1} holds, then
$\sum\limits_{t=0}^{+\infty}\left(A^{\top}\right)^{t}H
A^{t}-\sum\limits_{t=0}^{+\infty}\left(A^{\top}\right)^{t+1}H
A^{t+1}=\sum\limits_{t=0}^{+\infty}\left(A^{\top}\right)^{t}\Theta
A^{t}$. It follows from Parseval's Identity that $\displaystyle
H=\dfrac{1}{2\pi i} \oint_{\mathbf{\zeta}} g\left(z\right)^{\top}
\Theta g(z) \dfrac{1}{z}dz$ and it follows from  \eqref{matrix_eqn2}
and \eqref{theta_function} that $G=\Theta-T\left(\Theta\right)$.
Thus \eqref{G_Theta} holds. This completes the proof of Lemma
\ref{theta_condition}.
\end{proof}
\begin{lemma}\label{matrix_system}
If the system \eqref{problem_constraints}-\eqref{problem_IC} is
stable in the mean-square sense for $u_{t}\equiv 0$ then Condition
\ref{matrix_condition} holds.
\end{lemma}
\begin{proof}
Let us denote the $i$-th column of a matrix $D$ by $D_{.,i}$, let the matrices $H$ and $G$ be as in Condition \ref{matrix_condition} and let $h=\left[H_{.,1}^{\top},\ldots,H_{.,n}^{\top}\right]$, $\theta=\left[\Theta_{.,1}^{\top},\ldots,\Theta{.,n}^{\top}\right]$ and $g=\left[G_{.,1}^{\top},\ldots,G_{.,n}^{\top}\right]$. Let $\mathcal{A}_{1}=A^{\top}\otimes A^{\top}-I_{n^{2}}$ and $\mathcal{A}_{2}= -C^{\top}\otimes C^{\top}$, where $I_{n^{2}}$ is the $n^{2}\times n^{2}$ identity matrix. We can rewrite \eqref{matrix_eqn1}-\eqref{matrix_eqn2} as
\begin{eqnarray}
\left[
            \begin{array}{cc}
                                    \mathcal{A}_{1} & I_{n^{2}}\\
                                    \mathcal{A}_{2} & I_{n^{2}}\\
            \end{array}
\right]
\left[
            \begin{array}{c}
                                h\\
                                \theta\\
            \end{array}
\right]
            &=&
\left[
            \begin{array}{c}
                                0\\
                                g\\
            \end{array}
\right].
\label{matrix_eqn3}
\end{eqnarray}
Please notice that the system in \eqref{matrix_eqn3} would be
degenerate if and only if $\mathcal{A}_{1}=\mathcal{A}_{2}$;
however, this would require that $\mathcal{A}=A\otimes A+C\otimes
C=I_{n^{2}}$ which would violate the assumption that the matrix
$\mathcal{A}$ from \eqref{matrix_to_linear} satisfies
$\rho\left(\mathcal{A}\right)<1$ (which is equivalent to the
requirement that the system
\eqref{problem_constraints}-\eqref{problem_IC} be stable in the
mean-square sense for $u_{t}=0$). Therefore, if the mean-square
stability is satisfied, we can assume that the system in
\eqref{matrix_eqn3} always has a solution, $\left[h^{\top},
g^{\top}\right]^{\top}$. Hence matrices $H$ and $G$ exist (that is,
Condition \ref{matrix_condition} holds). This proves Lemma
\ref{matrix_system}.
\end{proof}
It follows from Lemma \ref{theta_condition} that
$G=\Theta-T\left(\Theta\right)$. Therefore, if we set $x_{t}=0$ and
$u_{t}=0$ for $t<0$, we can rewrite \eqref{problem_functional} as
\begin{eqnarray}
\Phi \left(u_{.}\right) & = & \sum\limits_{t=-\infty}^{+\infty} \mathbf{E} \, \mathcal{F}\left(x_{t},u_{t}\right)- \sum\limits_{t=-\infty}^{+\infty} \mathbf{E} \,  x_{t}^{*} T\left(\Theta\right) x_{t} \label{problem_functional2}
\end{eqnarray}
Let the matrix-valued function $\Pi : \mathbb{C} \mapsto \mathbb{C}^{m \times m}$ be defined by
\begin{eqnarray}
\Pi\left(z\right) = b^{\top}g\left(\overline{z}\right)^{\top}\Theta g\left(z\right)b + b^{\top}g\left(\overline{z}\right)^{\top}\gamma + \gamma^{\top}g\left(z\right)b + \Gamma,\label{matrix_PI}
\end{eqnarray}
and let
\begin{eqnarray}
\left(\hat{u}\left(z\right),R\hat{u}\left(z\right)\right)
&=&\frac{1}{2\pi i}\oint_{\mathbf{\zeta}}\hat{u}\left(z\right)^{*}\Pi\left(z\right)\hat{u}\left(z\right)\frac{1}{z}dz,\label{squared_term}\\
\left(r, \hat{u}\left(z\right)\right) &=& \frac{1}{2\pi i}\oint_{\mathbf{\zeta}}\mathbf{E}\, \overline{z}a^{\top}g\left(\overline{z}\right)^{\top}\left[G g\left(z\right)b + \gamma\right]\hat{u}\left(z\right)\frac{1}{z}dz,\label{linear_term}\\
\rho &=& \frac{1}{2\pi i}\oint_{\mathbf{\zeta}}\mathbf{E}\, a^{\top}g\left(\overline{z}\right)^{\top} G g\left(z\right)a\frac{1}{z}dz.\label{constant_term}
\end{eqnarray}
It follows from \eqref{squared_term}-\eqref{constant_term} and Parseval's identity that we can rewrite \eqref{problem_functional2} as
\begin{eqnarray}
\Phi\left(u_{.}\right) &=& \left(\hat{u}\left(z\right),R \hat{u}\left(z\right)\right) + \left(r, \hat{u}\left(z\right)\right) + \rho.\label{quadratic_form}
\end{eqnarray}
Thus. $\Phi\left(u_{.}\right)$ is a quadratic form in
$\hat{u}\left(.\right)$. Consider the deterministic control problem
below.\par \par Minimize
\begin{eqnarray}
\Phi_{1} \left(u_{.}\right) & = & \sum\limits_{t=0}^{+\infty} \left[ y_{t}^{*} \Theta x_{t} + 2\Re \, y_{t}^{*} \gamma u_{t} + u_{t}^{*} \Gamma u_{t}\right]\label{deterministic_functional}
\end{eqnarray}
over the set
\begin{eqnarray}
\mathbf{U} & = & \left\{ u_{t} \in \mathbb{R}^{m}: \sum\limits_{t=0}^{+\infty} \left|u_{t}\right|^{2} < +\infty\right\}\label{deterministic_feasible}
\end{eqnarray}
subject to
\begin{eqnarray}
y_{t+1} &=& A y_{t} + b u_{t}, \qquad t=0,1,2,\ldots \label{deterministic_constraints}\\
y_{0} &=& \mathbf{E}\, a. \label{deterministic_IC}
\end{eqnarray}
Here $y_{t}$ is an $n$-vector of states and $u_{t}$ is an $m$-vector of controls. Let matrices $G$, $\gamma$, $\Gamma$, $A$, and $b$ and the vector $a$ have the same properties as in the stochastic optimization problem \eqref{problem_functional}-\eqref{problem_IC} above, and let the matrix $\Theta$ be such that \eqref{G_Theta} is satisfied. Using Parseval's identity, we can rewrite \eqref{deterministic_functional} as $\Phi_{1}\left(u\left(.\right)\right) = \left(\hat{u}\left(.\right), R_{1} \hat{u}\left(.\right)\right)+\left(r_{1}, \hat{u}\left(.\right)\right)+\rho_{1}$, where
\begin{eqnarray}
\left(\hat{u}\left(z\right),R_{1} \hat{u}\left(z\right)\right)
&=&\frac{1}{2\pi i}\oint_{\mathbf{\zeta}}\hat{u}\left(z\right)^{*}\Pi\left(z\right)\hat{u}\left(z\right)\frac{1}{z}dz,\label{squared_term2}\\
\left(r_{1}, \hat{u}\left(z\right)
\right) &=& \frac{1}{2\pi i}\oint_{\mathbf{\zeta}}\mathbf{E}\, \overline{z}a^{\top}g\left(\overline{z}\right)^{\top}\left[G g\left(z\right)b + \gamma\right]\hat{u}\left(z\right)\frac{1}{z}dz,\label{linear_term2}\\
\rho_{1} &=& \frac{1}{2\pi i} \oint_{\mathbf{\zeta}}\mathbf{E}\,
a^{\top}g\left(\overline{z}\right)^{\top} G
g\left(z\right)a\frac{1}{z}dz.\label{constant_term2}
\end{eqnarray}
\begin{theorem}\label{supplementary_theorem}
An optimal control $u_{t}^{o}$ for the stochastic optimization problem \eqref{problem_functional}-\eqref{problem_IC} exists if and only if an optimal control for the deterministic optimization problem \eqref{deterministic_functional}-\eqref{deterministic_IC} exists. Furthermore, if \eqref{necessary_condition_strict} holds then the optimal controls in optimization problems \eqref{problem_functional}-\eqref{problem_IC} and \eqref{deterministic_functional}-\eqref{deterministic_IC} are identical and unique to within equivalence.
\end{theorem}
\begin{proof}
Note that, the necessary and sufficient conditions for the existence
of optimal $u^{o}$ that minimizes the quadratic form $\left(u, R
u\right) + \left(r, u\right) + \rho$ depend on $R$ and $r$.
Moreover, the optimal $u^{o}$, when it exists, is given by the
solution to $R u^{o} + r = 0$. We can see from
\eqref{squared_term}-\eqref{constant_term} and
\eqref{squared_term2}-\eqref{constant_term2} that $R_{1}=R$ and
$r_{1}=r$ for the functionals $\Phi$ and $\Phi_{1}$. It therefore
follows that the solution to
\eqref{problem_functional}-\eqref{problem_IC} exists if and only if
the solution to
\eqref{deterministic_functional}-\eqref{deterministic_IC} exists.
Furthermore, if \eqref{necessary_condition_strict} holds, it follows
from the results from \cite{Yakubovich74,Yakubovich75}, that the
optimal control for
\eqref{deterministic_functional}-\eqref{deterministic_IC} exists and
is unique. This completes the proof of Theorem
\ref{supplementary_theorem}.
\end{proof}
It follows that if the optimization problem
\eqref{deterministic_functional}-\eqref{deterministic_IC} has an
optimal solution then \eqref{necessary_condition} must hold.
Furthermore if \eqref{necessary_condition_strict} holds, it follows
that the solution exists and is unique (up to to within
equivalence). Hence the proof for Theorem \ref{main_theorem} follows
from Theorem \ref{supplementary_theorem}.
\begin{remark}
If $C = 0$ in the problem \eqref{problem_functional}-\eqref{problem_IC}, then the requirement that the system \eqref{problem_constraints}-\eqref{problem_IC} be stable in the mean-square sense for $u_{t} = 0$ will be equivalent to requiring that the matrix $A$ satisfy $\rho(A) < 1$. In addition, if we set $\Theta = G$ then \eqref{G_Theta} holds. Therefore Condition 6.1 is satisfied and $\mathcal{F}\left(x, u\right) = x^{*} G x + 2 \Re x^{*} \gamma x + u^{*}\Gamma u$, and Theorem \ref{matrix_condition} will be the same as the results from \cite{Yakubovich74,Yakubovich75} with $x_{0} = \mathbf{E}\, a$.
\end{remark}
\subsection{Numerical Algorithm}\label{numerics}
In this section we provide a Matlab code that takes matrices $G$, $A$ and $C$ as inputs, then checks if the system is stable in the mean-square sense. If it is stable, the program calculates matrices $\Theta$ and $H$.
\begin{verbatim}
%--------------------------------------------------------------------
%FILE NAME: numerics.m
%DESCRIPTION: Check if the discrete-time Linear Quadratic Control
% Problem is Solvable. That is, we calculate H and Theta
% that satisfy: {A'HA-H+Theta=0, Theta-C'HC-G=0}
%INPUTS: Matrices G, A, C
%OUTPUT: Matirx Theta, H
%--------------------------------------------------------------------
function [Theta, H] = numerics1(G, A, C)
%Verify that the inputs are all square matrices of the same dimension
s1=size(G); s2=size(A); s3=size(C);
if((s1(1)~=s2(1))|(s1(1)~=s3(1))|(s2(1)~=s3(1))|...
    (s1(2)~=s2(2))|(s1(2)~=s3(2))|(s2(2)~=s3(2))|...
      (s1(1)~=s1(2))|(s2(1)~=s2(2))|(s3(1)~=s3(2)))
  disp('ERROR! Dimension Mismatch');
  return;
end;
%--------------------------------------------------------------------
%Get the symmetric part of G
G=0.5*(G+G');
%--------------------------------------------------------------------
%Verify that the spectral radius of A is less than 1
if(max(abs(eig(A))) >= 1)
  disp('Matrix A is not convergent');
  return;
end;
%--------------------------------------------------------------------
%Verify that the system is Exponential Bounded in the mean-square sense
Big_A = kron(A,A)+kron(C,C);
if(max(abs(eig(Big_A))) >= 1)
  disp('The system is not EMS stable');
  return;
end;
%--------------------------------------------------------------------
%Solve the system, i.e. Calculate matrices H and Theta
A1=kron(A',A')-eye(size(A').^2);
B1=eye(size(A').^2);
A2=-kron(C',C');
B2=eye(size(A').^2);
M=[A1, B1; A2, B2];
v=[zeros(size(G(:)));G(:)];
solution=M\v;
theta=solution(length(solution)/2+1:length(solution));
h=solution(1:length(solution)/2);
%--------------------------------------------------------------------
%Return outputs H and Theta and Terminate the program
Theta=reshape(theta,size(A));
H=reshape(h,size(A));
\end{verbatim}
%
\nocite{*}
\bibliographystyle{plain}
\bibliography{frequency10}

\begin{thebibliography}{10}

\bibitem{Andreev72}
V.~A. Andreev.
\newblock The synthesis of optimal controls for inhomogeneous linear systems
  with a quadratic quality criterion.
\newblock {\em Sibirsk. Matem. Zh.}, 13(3):698--702, 1972.

\bibitem{Andreev_et_al71}
V.~A. Andreev, Yu.~F. Kazarinov, and V.~A. Yakubovich.
\newblock Synthesis of optimal controls for linear nonhomogeneous systems in
  problems of minimization of quadratic functionals.
\newblock {\em Soviet Mathematics Doklady}, 199(2):257--261, 1971.

\bibitem{Andreev_et_al72}
V.~A. Andreev, Yu.~F. Kazarinov, and V.~A. Yakubovich.
\newblock On the synthesis of optimal controls in the problem of minimization
  of a quadratic functional.
\newblock {\em Soviet Mathematics Doklady}, 202(6):1247--1250, 1972.

\bibitem{Andreev_Shepelyavyi72}
V.~A. Andreev and A.~I. Shepelyavyi.
\newblock Synthesis of optimal controls for pulse--amplitude systems in the
  problem of minimization of the mean value of a quadratic functional.
\newblock {\em Sibirsk. Matem. Zh.}, 14(2):250--276, 1972.

\bibitem{Barabanov07}
N.~E. Barabanov.
\newblock Kalman--{Y}akubovich lemma in general finite dimensional case.
\newblock {\em Int. J. Robust Nonlinear Control}, 17:369--386, 2007.

\bibitem{Dokuchaev84}
N.~G. Dokuchaev.
\newblock A frequency criterion for the existence of an optimal control for
  {I}t\^o equations.
\newblock {\em Vestnik Leningrad University. Mathematics}, 16:41--47, 1984.

\bibitem{Korenevskii86_2}
D.~G. Korenevskii.
\newblock Algebraic criteria and sufficient conditions for asymptotic stability
  and boundedness with probability 1 of solutions of a system of linear
  stochastic difference equations.
\newblock {\em Ukr. Mat. Zh.}, 38(4):447--452, 1986.

\bibitem{Korenevskii86_3}
D.~G. Korenevskii.
\newblock Matrix criteria and sufficient conditions for asymptotic stability
  and boundedness with probability one of solutions of linear stochastic
  difference equations.
\newblock {\em Doklo Akad. Nauk SSSR}, 290(6):1294--1298, 1986.

\bibitem{Korenevskii92}
D.~G. Korenevskii.
\newblock Equivalence of spectral and coefficient criteria for the mean--square
  asymptotic stability of solutions of systems of linear stochastic
  differential and difference equations.
\newblock In {\em Mathematical Methods for the Investigation of Applied
  Problems of Dynamics of Solids Carrying Liquid}, pages 47--52. Institute of
  Mathematics, Ukrainian Academy of Sciences, Kiev, 1992.

\bibitem{Korenevskii98}
D.~G. Korenevskii.
\newblock Criteria for the mean-square asymptotic stability of solutions of
  systems of linear stochastic difference equations with continuous time and
  delay.
\newblock {\em Ukrainian Mathematical Journal}, 50(8):1073--1081, 1998.

\bibitem{Korenevskii00}
D.~G. Korenevskii.
\newblock Relationship between spectral and coefficient criteria of
  mean--square stability for systems of linear stochastic differential and
  difference equations.
\newblock {\em Ukrainian Mathematical Journal}, 52(2):260--266, 2000.

\bibitem{Levit_Yakubovich72}
M.~V. Levit and V.~A. Yakubovich.
\newblock Algebraic criterion for stochastic stability of linear systems with
  parametric action of the white noise type.
\newblock {\em J. Appl. Math. Mech.}, 36:130--136, 1972.
\newblock Prikl. Mat. Mekh. 36, 142--148 (1972).

\bibitem{Likhtarnikov_et_al96}
A.~L. Likhtarnikov, N.~E. Barabanov, G.~A. Leonov, A.~H. Gelig, A.~S. Matveev,
  V.~B. Smirnova, and A.~L. Fradkov.
\newblock Frequency domain theorem ({Y}akubovich--{K}alman lemma) in the
  control theory.
\newblock {\em Automation and Remote Control}, 57(10):3--40, 1996.

\bibitem{Lur'e63}
A.~I. Lur'e.
\newblock A minimum quality criterion for control systems.
\newblock {\em Izv. Akad. Nauk SSSR, Otd. Tekhn. Nauk, Tekhnicheskaya
  Kibernetika}, 4:140--146, 1963.

\bibitem{Morozan68}
T.~Morozan.
\newblock Stability of stochastic discrete systems.
\newblock {\em J. Math. Anal. Appl.}, 23(1):1--9, 1968.

\bibitem{Nalitolela10}
P.~S. Nalitolela.
\newblock Frequency {C}riteria of {O}ptimal {C}ontrol {E}xistence for
  {S}tochastic {M}odels.
\newblock Master's thesis, Trent University, Peterborough, Ontario, Canada,
  January 2010.

\bibitem{Pakshin06}
P.~V. Pakshin and V.~A. Ugrinovskii.
\newblock Stochastic problems of absolute stability.
\newblock {\em Automation and Remote Control}, 67(11):1811--1846, 2006.
\newblock Original Russian Text published in {\it Avtomatika i Telemekhanika},
  No. 11, pp. 122--158, 2006.

\bibitem{Ryashko_Schurz96}
B.~L. Ryashko and H.~Schurz.
\newblock Mean square stability analysis of some linear stochastic systems.
\newblock {\em Dynamics Systems Appl.}, 6(2):165--190, 1996.

\bibitem{Saitoh_et_al03}
Saburou Saitoh, Vu~Kim Tuan, , and Mashiro Yamamoto.
\newblock Convolution inequalities and applications.
\newblock {\em Journal of Inequalities in Pure and Applied Mathematics},
  4(3):Article 50, 2003.

\bibitem{Szego_Kalman63}
G.~Szeg\H{o} and R.~E. Kalman.
\newblock Sur la stabilite absolue d'un sisteme d'equations aux differences
  finies.
\newblock {\em Comptes Rendus del Academie des Sciences, Paris}, 257:388--390,
  1963.

\bibitem{Willems73}
J.~L. Willems.
\newblock Mean square stability criteria for linear white noise stochastic
  systems.
\newblock {\em Probl. Contr. Inf. Theory}, 2(3--4):199--217, 1973.

\bibitem{Yakubovich62}
V.~A. Yakubovich.
\newblock The solution of certain matrix inequalities encountered in automatic
  control theory.
\newblock {\em Soviet Mathematics Doklady}, 143(6):1304--1307, 1962.

\bibitem{Yakubovich70_2}
V.~A. Yakubovich.
\newblock Solution of one algebraic problem encountered in the control theory.
\newblock {\em Dokl. Akad. Nauk SSSR}, 193(1):57--60, 1970.

\bibitem{Yakubovich73_2}
V.~A. Yakubovich.
\newblock Frequency domain theorem in control theory.
\newblock {\em Siberian Mathematical Journal}, 14(2):265--289, 1973.

\bibitem{Yakubovich74}
V.~A. Yakubovich.
\newblock A frequency theorem for the case in which the state and control
  vectors are {H}ilbert spaces, with an application to some problems in the
  synthesis of optimal controls.
\newblock {\em Sibirsk. Math. Zh.}, 15(3):639--668, 1974.

\bibitem{Yakubovich75}
V.~A. Yakubovich.
\newblock A frequency theorem for the case in which the state and control
  vectors are {H}ilbert spaces, with an application to some problems in the
  synthesis of optimal controls, {II}.
\newblock {\em Sibirsk. Math. Zh.}, 16(5):1081--1102, 1975.

\end{thebibliography}
\end{document}